\documentclass{cimart}

\title{Graded identities for matrix algebras of order two over a finite field}

\authors{Diogo Diniz, Eduardo Pinto da Fonsêca and Luis Filipe Ramos}

\authorinfo[D. Diniz]{Federal University of Campina Grande, Brazil}{diogo@mat.ufcg.edu.br}

\authorinfo[E. Pinto da Fonsêca]{Federal University of Campina Grande, Brazil}{pintokimura@gmail.com}

\authorinfo[L. F. Ramos]{Federal University of Campina Grande, Brazil}{luis\_filipecg@hotmail.com}

\abstract{%
    Let $G$ be an arbitrary group and let $\mathbb{F}$ be a finite field. In this paper, we determine bases for the $T_G$-ideals of graded polynomial identities of the algebra $M_2(\mathbb{F})$ for all possible $G$-gradings. The bases obtained consist of finitely many non-trivial graded identities, and are finite whenever $G$ is finite.
    }



\VOLUME{34}
\YEAR{2026}
\ISSUE{2}
\NUMBER{14}
\DOI{https://doi.org/10.46298/cm.18621}

\editinfo{June 25, 2026}{September 10, 2026}{Tiago Macedo, Jaqueline Mesquita, Mariel Saez and Rafael Potrie}

\acknowledgments{This study was financed in part by the Coordena\c{c}\~ao de Aperfei\c{c}oamento de Pessoal de N\'ivel Superior - Brasil (CAPES) - Finance Code 001. D.~Diniz acknowledges the support of the Brazilian National Council for Scientific and Technological Development (CNPq) through grant No.~304328/2022-7.

The authors would like to thank the referees for their valuable comments and suggestions, which helped improve the presentation of the paper.}

\begin{document}

\section{Introduction}
A fundamental development in the theory of PI-algebras is A. Kemer's solution to the Specht Problem for algebras over a field of characteristic zero. Given a concrete algebra, however, it is a very difficult problem to determine a finite basis for its polynomial identities. Let $M_2(\mathbb{K})$ be the matrix algebra of order two over the field $\mathbb{K}$. When $\mathbb{K}$ has characteristic zero, Yu. P. Razmyslov provided a finite basis for the identities of $M_2(\mathbb{K})$ in \cite{R} and V. Drensky provided a minimal basis, consisting of two identities, in \cite{D}. Over an infinite field of characteristic different from two, analogous results were obtained by P. Koshlukov in \cite{Ko}. A basis for the identities of the matrix algebras of order two over a finite field was determined by Yu. N. Mal'tsev and E. N. Kuz'min in \cite{MK}.

An important concept in Kemer's theory of $T$-ideals in characteristic zero is the notion of graded polynomial identities for graded algebras. It is natural to extend Specht's Problem to the context of graded algebras. In the case of algebras over fields of characteristic zero, a positive answer has been given by E. Aljadeff and A. Kanel-Belov in \cite{AKB} and by I. Sviridova in \cite{Sv}. O. M. Di Vincenzo in \cite{DiVincenzo} and P. Koshlukov and S. S. Azevedo in \cite{KA2} determined bases for the graded identities of $M_2(\mathbb{K})$ in characteristic zero and when $\mathbb{K}$ is infinite, respectively. When $\mathbb{K}$ is a finite field of characteristic different from two, the $\mathbb{Z}_2$-gradings on  $M_2(\mathbb{K})$ were described and finite bases for the graded polynomial identities were determined in \cite{KA}.

Gradings on matrix algebras over an algebraically closed field were classified by Yu. A. Bahturin and M. Zaicev in \cite{BZ}. Over an arbitrary field, the group gradings on matrix algebras of order two were classified in \cite{KBD}. More generally, analogous results concerning matrix algebras of prime order were obtained in \cite{BDW} and \cite{DG}. In the present paper we determine bases for the graded identities for the matrix algebra of order two with an arbitrary group grading. We also remark that the bases obtained are finite when the grading group is finite. The existence of a finite basis in this case follows from the main result of \cite{BCD}.

The paper is organized as follows. In Section~\ref{prelim}, we present the basic definitions and results on graded algebras and graded polynomial identities that will be used throughout the paper. In Section~\ref{results}, we present the main results, namely, bases for the graded polynomial identities of $M_2(\mathbb{F})$, where $\mathbb{F}$ is a finite field, for all possible gradings. There are six types of gradings to consider, two of which were studied in \cite{KA}. In this section, we determine bases for the graded identities of the remaining gradings. 

\section{Preliminaries}\label{prelim}

Let $G$ be a group with identity element $e$. A $G$-grading on a vector space $A$ over the field $\mathbb{K}$ is a decomposition $A=\oplus_{g\in G}A_g$ of $A$ as a direct sum of the family $\{A_g\mid g\in G\}$ of subspaces. If $A$ is an algebra, we further assume that $A_gA_h\subseteq A_{gh}$, for all $g,h\in G$. An element $a$ of $A$ is called homogeneous if there exists $g\in G$ such that $a\in A_g$. Note that if $a\neq 0$, then such $g$ is unique. In this case we say that $g$ is the degree of $a$. Henceforth, $||a||$ denotes the degree of $a$ whenever $a$ is a non-zero homogeneous element. The support of $A=\oplus_{g\in G}A_g$ is the set $\mathrm{supp}\, A:=\{g\in G\mid A_g\neq 0\}$. As an example of a $G$-grading on $A$ we may consider the trivial grading, i.\,e., the grading where $A_e=A$ and $A_g=0$ for all $g\in G\setminus \{e\}$.

Let $A$ and $B$ be $G$-graded algebras. A homomorphism $\varphi:A \rightarrow B$ of algebras is a homomorphism of graded algebras if $\varphi(A_g)\subseteq B_g$ for all $g\in G$. Henceforth, we denote by $[A]$ the isomorphism class of $A$ as a $G$-graded algebra. 

An ideal (resp.\, subalgebra) $I$ of $A$ is called a graded ideal (resp.\, graded subalgebra) if $I=\oplus_{g\in G}(I\cap A_g)$. We say that $A$ is a graded simple algebra if $A^2\neq 0$ and $0$ and $A$ are the only graded ideals of $A$. We call $A$ a graded division algebra if it is unital and every non-zero homogeneous element of $A$ is invertible.
Let $D$ be a graded division algebra and let $A=M_n(D)$. Given a tuple $\gamma=(g_1,\dots, g_n)$, the decomposition $A=\oplus_{g\in G} A_g$ where \[A_g:=\left( \begin{array}{ccc}
        D_{g_1^{-1}gg_1} & \cdots & D_{g_1^{-1}gg_n}\\
        \vdots & \ddots & \vdots\\
        D_{g_n^{-1}gg_1} & \cdots & D_{g_n^{-1}gg_n}\\
\end{array}\right)\]
is a $G$-grading on $A$.

Note that if $d\in D_h$ and $dE_{ij}$ is the matrix with $d$ in the $(i,j)$-th position and $0$ elsewhere, then $||dE_{ij}||=g_ihg_j^{-1}$. Then $A$ is a graded simple algebra; moreover, every graded simple algebra that is graded left (or right) Artinian is isomorphic to an algebra of this type; see \cite[Theorem 1.58]{NV}. Henceforth, we denote the algebra $M_n(D)$ with this grading by $M_n(D)(\gamma)$. In order to simplify the notation, when it is not necessary to specify $\gamma$, we use the notation $M_n(D)$ instead of $M_n(D)(\gamma)$.

A graded (left) $A$-module is an $A$-module $M=\oplus_{g\in G}M_g$ with a grading as a vector space such that $A_gM_h\subseteq M_{gh}$ for all $g,h\in G$. A submodule $N$ of $M$ is a graded submodule if it is a graded subspace. The annihilator of $M$ in $\mathcal{A}$ is the set \[\textrm{Ann}_G(M)=\{a\in A\mid ax=0\, \forall x\in M\}.\] We say that $M$ is a graded simple module if $AM\neq 0$ and $0$ and $M$ are the only graded submodules.

\begin{definition}
Let $A$ be a $G$-graded algebra. The $G$-graded Jacobson radical of $A$ is the set \[J^{gr}(A)=\bigcap \textrm{Ann}_G(M),\]
where the intersection runs over all the $G$-graded  simple left $A$-modules. We define $J^{gr}(A)=A$ if the algebra $A$ admits no graded simple modules.
\end{definition}

\begin{remark}\label{remarkgradedJ}
As in the ordinary case, if $A$ is finite-dimensional, then $J^{gr}(A)$ is a nilpotent $G$-graded ideal that contains every nilpotent graded ideal. This fact will be important in the proofs of the main results. The proof is analogous to the ordinary case and will be omitted.
\end{remark}

\begin{lemma}[see {\cite[Lemma 2.3]{DGK}}] \label{idempotent}
	Let $A$ be a finite-dimensional graded algebra such that $J^{gr}(A) = 0$. If $I$ is a graded ideal of $A$, there exists a homogeneous central
idempotent $e$ such that $I=Ae$.
\end{lemma}

We recall that the unitization of an $\mathbb{F}$-algebra $A$ is the vector space $A^{\sharp}=\mathbb{F}\times A$ with the multiplication \[(\lambda, x)(\mu, y)=(\lambda \mu, \mu x+\lambda y+xy).\] We identify $\mathbb{F}$ with the subspace $\mathbb{F}\times \{0\}$ and $A$ with the ideal $\{0\}\times A$ and write $A^{\sharp}=\mathbb{F}\oplus A$. If $A=\oplus_{g\in G} A_g$ is $G$-graded then we consider the grading on $A^{\sharp}$ such that $(A^{\sharp})_e=\mathbb{F}\oplus A_e$ and $(A^{\sharp})_g=A_g$ for all $g\in G\setminus \{e\}$.

\begin{lemma}\label{jgrunitization}
Let $A$ be a finite-dimensional $G$-graded algebra and let $A^{\sharp}$ be its unitization. Then $J^{gr}(A)=J^{gr}(A^{\sharp})$.
\end{lemma}
\begin{proof}
It follows from Remark \ref{remarkgradedJ} that $J^{gr}(A^{\sharp})\subseteq J(A^{\sharp})$; therefore $J^{gr}(A^{\sharp})$ is a nilpotent ideal. This implies that $J^{gr}(A^{\sharp})\subseteq A$. As a consequence $J^{gr}(A^{\sharp})\subseteq J^{gr}(A)$. Since $J^{gr}(A)$ is a nilpotent graded ideal of $A^{\sharp}$, Remark \ref{remarkgradedJ} implies that $J^{gr}(A)\subseteq J^{gr}(A^{\sharp})$.
\end{proof}

\begin{corollary}\label{jgrunital}
If $A$ is a finite-dimensional graded algebra with $J^{gr}(A) = 0$ then $A$ is a unital algebra. 
\end{corollary}
\begin{proof}
Let $A^{\sharp}$ be the unitization of $A$. Lemma \ref{jgrunitization} implies that $J^{gr}(A^{\sharp}) = 0$. Since $A$ is a graded ideal of $A^{\sharp}$, the result follows from Lemma \ref{idempotent}.
\end{proof}

Note that Corollary \ref{jgrunital} implies that \cite[Theorem 2.4]{DGK} holds without the hypothesis that the algebra is unital. Hence, we have the following result:

\begin{theorem}\label{decomposição estrutural de álgebras semi primas graduadas}
	Let $A$ be a finite-dimensional graded algebra. If $J^{gr}(A) = 0$ then $A$ is a direct sum of finitely many graded simple algebras.
\end{theorem}

\begin{definition}
A graded ideal $J$ of $A$ is graded-nil if every homogeneous element in $J$ is nilpotent.
\end{definition}

The following result will be useful to prove the main results of the paper. Its proof is similar to the proof in the ordinary case given in \cite{GZ}. For the convenience of the reader, we present the proof in the graded case here.

\begin{lemma}\label{lifting}
		Let $A$ be a unital $G$-graded algebra and let $J$ be a graded-nil ideal of $A$. Then:  
		
		\begin{enumerate}
			\item Every homogeneous idempotent of $A/J$ can be lifted to a homogeneous idempotent of $A$.
			
			\item Every finite set of orthogonal homogeneous idempotents of $A/J$ can be lifted to a set of orthogonal homogeneous idempotents of $A$.
			
			\item Every set of homogeneous $n\times n$ matrix units of $A/J$ can be lifted to a set of homogeneous $n\times n$ matrix units of $A$. 
		\end{enumerate}
	\end{lemma}
	
	\begin{proof}

 $(1)$ Let $y\in A/J$ be a homogeneous idempotent. Then $y$ lies in $(A/J)_e$. Hence, there exists an element $a\in A_{e}$ such that $\overline{a}=y$, where $\overline{a}=a+J$. Note that $a-a^{2}\in J$. Since $a-a^{2}\in J_e$ and $J$ is graded-nil, there exists an integer $k\geq 1$ such that $(a-a^{2})^{k}=0$. We expand $(a-a^{2})^{k}$ and obtain a polynomial $f(x)\in \mathbb{F}[x]$ such that 
 	\[0=(a-a^{2})^{k}=a^{k}-a^{k+1}f(a).\] 	By induction on $t$, we conclude that $a^{k}=a^{k+t}f^{t}$ for all $t\in\mathbb{N}$, in particular $a^{2k}f^{k}=a^{k}$. Note that $x=(af)^{k}\in A_{e}$. We have
 	$$x^{2}=(af)^{2k}=a^{2k}f^{k}f^{k}=a^{k}f^{k}=(af)^{k}=x.$$
 	Therefore, $x$ is a homogeneous idempotent of $A$. Moreover,
 	$$\overline{x}=\overline{(af)^{k}}=\overline{a}^{k}\overline{f^{k}}=y^{k}\overline{f^{k}}=(y^{2})^{k}\overline{f^{k}}=\overline{a^{2k}}\overline{f^{k}}=\overline{a^{2k}f^{k}}=\overline{a^{k}}=y^{k}=y.$$
 	
 	$(2)$ We prove this assertion by induction on the number $n$ of orthogonal idempotents to be lifted. For $n=1$, the result follows from $(1)$. Let us assume the result holds for $n$ homogeneous orthogonal idempotents and let $y_{1},\ldots,y_{n+1}$ be a set of homogeneous orthogonal idempotents of $A/J$. Then $y_{1},\ldots,y_{n}\in (A/J)_{e}$ and there exist $x_{1},\ldots,x_{n}\in A_{e}$ orthogonal idempotents such that $\overline{x_{i}}=y_{i}$ for all $i=1,\ldots,n$. Let $s=\sum_{i=1}^{n}x_{i}\in A_{e}$. Moreover, by $(1)$ the idempotent $y_{n+1}$ can be lifted to an idempotent $x\in A_{e}$. Note that $\overline{sx}=0$, i.\,e., $sx\in J$. Since $sx$ is homogeneous and $J$ graded-nil, we have $(sx)^{k}=0$, for some integer $k\geq 1$. Then $1-sx$ is invertible, and 
    \begin{equation}\label{inverse}
    (1-sx)^{-1}=(sx)^{k-1}+\cdots+sx+1.
    \end{equation}
    Let $x_{n+1}=(1-sx)x(1-sx)^{-1}(1-s)\in A_{e}$. We claim that $x_{n+1}$ is the desired idempotent. Since $s$ and $x$ are idempotents, direct computations yield
 	\[(1-s)(1-sx)x=(1-sx)x.\]
 	The equality above and the definition of $x_{n+1}$ together with the fact that $x$ is an idempotent imply that $x_{n+1}^2=x_{n+1}$.
 Moreover, from (\ref{inverse}) it follows that $\overline{(1-sx)^{-1}}=1$; therefore,
 	$\overline{x_{n+1}}=y_{n+1}$.
 	Finally, note that $(1-s)s=s-s^{2}=s-s=0$, hence $x_{n+1}s=0$. Moreover,
 	$s(1-sx)x=sx-s^2x^2=0$. Hence, $sx_{n+1}=0=x_{n+1}s$. Since $sx_{i}=x_{i}s=x_{i}$, we conclude that $x_ix_{n+1}=0=x_{n+1}x_i$,  for all $1\leq i\leq n$.
 	
 	$(3)$ Let $E_{ij}\in A/J$, $1\leq i,j\leq n$ be homogeneous matrix units. There exist $g_1,\dots, g_n$ in $G$ such that $E_{ij}$ has degree $g_ig_j^{-1}$ for all $1\leq i,j\leq n$. It follows from $(2)$ that there exist homogeneous orthogonal idempotents $e_{11},\ldots,e_{nn} $ in $A$ such that $\overline{e_{ii}}=E_{ii}$ for all $i=1,\dots, n$. Let $a_{ij}$ be an element of $A_{g_{i}g_{j}^{-1}}$ such that $\overline{a_{ij}}=E_{ij}$. We define for $i\geq 2$
    \begin{equation*}
    e_{i1}=e_{ii}a_{i1}e_{11}\in A_{g_{i}g_{1}^{-1}}. 
    \end{equation*}
    Let $b_{i}=e_{11}-a_{1i}e_{i1}\in A_e$. We have
 	$\overline{b_{i}}=E_{11}-E_{1i}E_{i1}=E_{11}-E_{11}=0$. Hence, $b_{i}\in J_e$ is nilpotent and $1-b_{i}$ is an invertible element of $A_e$. Note that the element 
    \begin{equation*}
    e_{1i}:=(1-b_{i})^{-1}a_{1i}e_{ii}\in A_{g_{1}g_{i}^{-1}} 
    \end{equation*}
    lifts $E_{1i}$. Note that $a_{1i}e_{i1}=e_{11}-b_i$, hence 	
    \begin{equation}\label{ei1e1i}
 	e_{1i}e_{i1}=(1-b_i)^{-1}a_{1i}e_{i1} =(1-b_i)^{-1}(e_{11}-b_{i})=(1-b_i)^{-1}(1-b_{i})e_{11}=e_{11}.
    \end{equation}
 Note that $e_{ii}e_{i1}=e_{i1}$ and $e_{1i}e_{ii}=e_{1i}$; hence (\ref{ei1e1i}) implies that $e_{ii}-e_{i1}e_{1i}$ is an idempotent. Moreover,
 	\begin{align*}
 		\overline{e_{ii}-e_{i1}e_{1i}}=E_{ii}-E_{i1}E_{1i}=0,
 	\end{align*}
 	Therefore $e_{ii}-e_{i1}e_{1i}$ lies in $J_e$ and is also nilpotent. This implies that $e_{i1}e_{1i}=e_{ii}$. Let $e_{ij}:=e_{i1}e_{1j}$. Then $\{e_{i,j}\mid 1\leq i,j\leq n\}$ is a set of matrix units and $e_{ij}$ lifts $E_{ij}$ for all $1\leq i,j\leq n$.
	\end{proof}


\subsection{Varieties of Graded Algebras}

Let $X_G=\cup_{g\in G}X_{g}$, where $X_{g}=\{x_{i,g}\mid i\geq 1, g\in G\}$. We assume that $X_{g}\cap X_{h}=\emptyset$ whenever $g\neq h$. We denote by $\mathbb{K}\langle X_G\rangle$ the free $G$-graded algebra. Recall that $||x_{i,g}||=g$ for all $i\geq 1$ and for all $g\in G$. 

Let $A=\oplus_{g\in G}A_g$ be a $G$-graded algebra. Any map $\psi: X_G\rightarrow A$ such that $\psi(X_{g})\subseteq A_g$ for all $g\in G$ can be uniquely extended to a homomorphism $\Psi:\mathbb{K}\langle X_G\rangle\rightarrow A$ of graded algebras. Given $f=f(x_{1,g_1},\dots, x_{n,g_n})\in \mathbb{K}\langle X_G \rangle$, the image of $f$ under $\Psi$ depends only on the elements $a_i=\psi(x_{i,g_i})$, for $i=1,\dots, n$ and is denoted $f(a_1,\dots, a_n)$. Moreover, $f=0$, or simply $f$, is a graded polynomial identity for $A$ if $f(a_1,\dots, a_n)=0$ for all $a_1\in A_{g_1}, \dots, a_n\in A_{g_n}$. The set $Id_G(A)$ of all graded polynomial identities for $A$ is a $T_G$-ideal of $\mathbb{K}\langle X_G\rangle$, i.\,e., it is a graded ideal invariant under all endomorphisms of $\mathbb{K}\langle X_G\rangle$ as a graded algebra. We remark that if $G=\{e\}$ is the trivial group, then we are speaking of the ordinary identities of $A$. In this case, we omit the symbols $G$ and $e$ from the notation. Hence, the free algebra is denoted $\mathbb{K}\langle X \rangle$ and $X=\{x_1,x_2,\dots\}$.

Let $S\subseteq \mathbb{K}\langle X_G\rangle$. The variety $\mathfrak{V}$ determined by $S$ is the class of all graded algebras $A$ such that $S\subseteq Id_G(A)$. Thus, by a variety of graded algebras we understand the class of all graded algebras determined by some given set of graded polynomial identities. Given a class $\mathcal{A}$ of algebras we denote by $\mathrm{Var}(\mathcal{A})$ the smallest variety that contains $\mathcal{A}$. If $\mathcal{A}=\{A\}$ we denote it simply by $\mathrm{Var}(A)$. A set $S\subseteq \mathbb{K}\langle X_G\rangle$ is a basis for $Id_G(A)$ if this is the smallest $T_G$-ideal of $\mathbb{K}\langle X_G\rangle$ that contains $S$. This occurs if and only if $\mathrm{Var}(A)$ is the variety determined by $S$. 

The following concept will be important for the proofs of our main results:

\begin{definition}
A $G$-graded algebra $A$ is graded subdirectly irreducible if the intersection of all non-zero graded ideals is different from zero.
\end{definition}

In \cite{BY} the authors prove that graded algebras can be seen as universal algebras. Therefore, Birkhoff's Theorem, see \cite[Theorem 2.6]{J} implies the following proposition:

\begin{proposition}\label{sdirr}
Every $G$-graded algebra is a subdirect product of graded subdirectly irreducible algebras.
\end{proposition}

A variety $\mathfrak{V}$ of graded algebras is locally finite if every algebra in $\mathfrak{V}$ is locally finite. 
Next we prove that a variety of algebras over a finite field that satisfies a polynomial identity of the form (\ref{identityv}) is generated by its finite graded subdirectly irreducible algebras. First we prove that such varieties are locally finite. 

\begin{proposition}\label{locfin}
Let $\mathbb{F}$ be a finite field. Let $\mathfrak{V}$ be a variety of $G$-graded $\mathbb{F}$-algebras and assume there exists $f(x_1,\dots, x_r)\in \mathbb{F}\langle X \rangle$ such that every monomial in $f$ has degree greater than $r$ and every algebra of $\mathfrak{V}$ satisfies the ordinary polynomial identity 
\begin{equation}\label{identityv}
x_1\cdots x_r=f(x_1,\dots, x_r).
\end{equation}
Then $\mathfrak{V}$ is locally finite. In particular $\mathfrak{V}$ is generated by its finite algebras.
\end{proposition}
\begin{proof}
Let $\mathfrak{U}$ be the variety of ordinary algebras determined by (\ref{identityv}). Then if we consider the algebras in $\mathfrak{V}$ as ordinary algebras we have $\mathfrak{V}\subseteq \mathfrak{U}$. \cite[Proposition 2.6 and Theorem~2.2]{Lvov} imply that $\mathfrak{U}$ is a locally finite variety of ordinary algebras. In particular $\mathfrak{V}$ is a locally finite variety of graded algebras. The fact that $\mathfrak{V}$ is generated by its finite algebras follows from \cite[Proposition~1.1]{Lvov}.
\end{proof}

\begin{corollary}\label{varsdirr}
Let $\mathbb{F}$ be a finite field. Let $\mathfrak{V}$ be a variety of $G$-graded $\mathbb{F}$-algebras and assume there exists $f(x_1,\dots, x_r)\in \mathbb{F}\langle X \rangle$ such that every monomial in $f$ has degree greater than $r$ and every algebra of $\mathfrak{V}$ satisfies the ordinary polynomial identity 
\begin{equation*}
x_1\cdots x_r=f(x_1,\dots, x_r).
\end{equation*}
Then $\mathfrak{V}$  is generated by the finite graded subdirectly irreducible algebras it contains.
\end{corollary}
\begin{proof}
Let $\mathfrak{W}$ be the subvariety generated by the finite graded subdirectly irreducible algebras in $\mathfrak{V}$. Let $A$ be a finite algebra in $\mathfrak{V}$. Proposition \ref{sdirr} implies that $A$ is the subdirect product of a family $\{A_i\mid i\in I\}$ of graded subdirectly irreducible algebras. Since each $A_i$ is a homomorphic image of $A$, it follows that $A_i$ is a finite algebra in $\mathfrak{V}$. Hence, $A_i\in \mathfrak{W}$ for all $i\in I$. Therefore, $A\in \mathfrak{W}$. Proposition \ref{locfin} implies that $\mathfrak{V}\subseteq \mathfrak{W}$. The reverse inclusion is clear; therefore, $\mathfrak{V}=\mathfrak{W}$.
\end{proof}

\subsection{Gradings on \texorpdfstring{$M_2(\mathbb{K})$}{M2(K)}}

Let $G$ be a group with identity $e$. Henceforth, we assume that the support of the grading generates $G$. There are three elementary gradings to consider (one of which is the trivial grading) and three division gradings (two when $\mathrm{char}\, \mathbb{K}\neq 2$ and one when $\mathrm{char}\, \mathbb{K} = 2$). Next we rewrite the main result of \cite{KBD}.

\begin{theorem}\label{teorema do M2KGRADINS}
		Let $G$ be a group with identity element $e$, let $\mathbb{K}$ be a field, and $M = M_2(\mathbb{K})$. If the support of the grading generates $G$ then $M$ is isomorphic to one of the following gradings:
		
		\textnormal{(I)} Elementary gradings:
		
		\begin{enumerate}
			\item[(1)] The trivial grading, that is, $G=\{e\}$ and
			\[
			M_e = M.
			\]
			
			\item[(2)] $G=\langle g \rangle\cong \mathbb{Z}_2$ and
			\[
			M_e =
			\begin{pmatrix}
				\mathbb{K} & 0 \\
				0 & \mathbb{K}
			\end{pmatrix},
			\quad
			M_g =
			\begin{pmatrix}
				0 & \mathbb{K} \\
				\mathbb{K} & 0
			\end{pmatrix}.
			\]

            \item[(3)] $G=\langle g \rangle$, where $|G|>2$ and 
			\[
			M_e =
			\begin{pmatrix}
				\mathbb{K} & 0 \\
				0 & \mathbb{K}
			\end{pmatrix},
			\quad
			M_g =
			\begin{pmatrix}
				0 & \mathbb{K} \\
				0 & 0
			\end{pmatrix},
			\quad
			M_{g^{-1}} =
			\begin{pmatrix}
				0 & 0 \\
				\mathbb{K} & 0
			\end{pmatrix},
			\]
			\[
			M_h = 0 \text{ for } h \in G \setminus \{e,g,g^{-1}\},
			\]
			where $g \in G$ is an element of order greater than $2$.			
			
		\end{enumerate}
		
		\textnormal{(II)} Division gradings:

        \begin{enumerate}
        \item[(4)] $G=\langle g \rangle \cong \mathbb{Z}_2$, $\mathrm{char}\,\mathbb{K}\neq 2$ and 
			\[
			M_e =
			\left\{
			\begin{pmatrix}
				u & v \\
				bv & u
			\end{pmatrix}
			\;\middle|\; u,v \in \mathbb{K}
			\right\},
			\quad
			M_g =
			\left\{
			\begin{pmatrix}
				u & v \\
				-bv & -u
			\end{pmatrix}
			\;\middle|\; u,v \in \mathbb{K}
			\right\},
			\]
			where $b \in \mathbb{K} \setminus \mathbb{K}^2$.

		\item[(5)] $G=\langle g \rangle \cong \mathbb{Z}_2$, $\mathrm{char}\,\mathbb{K}= 2$ and
        
        \[ 
		M_e =
		\left\{
		\begin{pmatrix}
			u & u+v \\
			b(u+v) & v
		\end{pmatrix}
		\;\middle|\; u,v \in \mathbb{K}
		\right\},\]
        
		\[
		M_g =
		\left\{
		\begin{pmatrix}
			bu+v & u \\
			v & bu+v
		\end{pmatrix}
		\;\middle|\; u,v \in \mathbb{K}
		\right\},
		\]
		where $b \in \mathbb{K} \setminus \{ x^2 + x \mid x \in \mathbb{K} \}$.

         \item[(6)] $G=\langle g, h\rangle \cong \mathbb{Z}_2\times \mathbb{Z}_2$, $\mathrm{char}\,\mathbb{K}\neq 2$ and
			\[
			M_e = \mathbb{K} I_2, \quad M_g = \mathbb{K}X, \quad M_h = \mathbb{K}Y, \quad M_{gh} = \mathbb{K}XY,\]
			where $X,Y$ are invertible matrices such that
			\[
			X^2, Y^2 \in \mathbb{K} I_2
			\quad \text{and} \quad
			XY = -YX.
			\]
        \end{enumerate}   

    	\end{theorem}

If $A=\oplus_{g\in G}A_g$ is a $G$-graded algebra and $\psi:G\rightarrow \tilde{G}$ is a monomorphism of groups then $A^{\psi}=\oplus_{\tilde{g}\in \tilde{G}}A_{\tilde{g}}$, where $A_{\tilde{g}}=A_g$ if $\tilde{g}=\psi(g)$ and $A_{\tilde{g}}=0$ if $\tilde{g}\notin \psi(G)$ is a $\tilde{G}$-grading on~$A$. 

\begin{remark}\label{r13}   	
Let $S\subseteq \mathbb{K}\langle X_G\rangle$ and let $\tilde{S}=\tilde{S_1}\cup \tilde{S_2}$, where \[\tilde{S_1}=\{f(x_{1,\psi(g_1)},\dots, x_{n,\psi(g_n)})\mid f(x_{1,g_1},\dots, x_{n,g_n})\in S\}\] and \[\tilde{S_2}=\{x_{1,\tilde{h}}\mid \tilde{h}\in \tilde{G}\setminus \psi(G)\}.\] If $S$ is a basis for $Id_G(A)$ then $\tilde{S}$ is a basis for $Id_{\tilde{G}}(A^{\psi})$.
\end{remark}

The previous remark allows us to simplify the notation in order to determine bases for the graded identities of the matrix algebra of order two. For the trivial grading we consider $G=\{e\}$ and the identities are simply the ordinary identities. For the elementary grading in (2) and the division gradings in (4) and (5), we assume that $G=\mathbb{Z}_2$. For the division grading in (6) we assume that $G=\mathbb{Z}_2\times \mathbb{Z}_2$. For the elementary grading (3) we assume that $G=\langle g \rangle$ is a cyclic group of arbitrary order, with multiplicative notation. 

\section{Main Results}\label{results}

In this section $\mathbb{F}$ denotes a field with $q$ elements. We present here the main results of the paper, namely, a basis for the graded identities of $M_2(\mathbb{F})$ with an arbitrary group grading. We assume that the support of the grading generates the group. 

A basis for the ordinary identities of $M_2(\mathbb{F})$ was determined by Yu. N. Mal'tsev and E. N. Kuz'min in \cite{MK}. We will need the following notation:

\begin{notation}
Let $a,b$ be elements of an algebra $A$. We define $[a,b]:=ab-ba$ and $a\circ b:=ab+ba$.
\end{notation}

\begin{theorem}[see \cite{MK}]
Let $\mathbb{F}$ be a field with $q$ elements. The ideal of identities of the algebra $M_2(\mathbb{F})$ is generated by the polynomials 
	 $$(x_{1}-x_{1}^{q})(x_{2}-x_{2}^{q^{2}})(1-[x_{1},x_{2}]^{q-1})=0$$
     and
	 $$(x_{1}-x_{1}^{q})\circ(x_{2}-x_{2}^{q})-((x_{1}-x_{1}^{q})\circ(x_{2}-x_{2}^{q}))^{q}=0.$$
\end{theorem}

Henceforth, we consider the non-trivial gradings on $M_2(\mathbb{F})$. When $\mathrm{char}\, \mathbb{F}\neq 2$, bases for graded identities for the gradings (2) and (4) in Theorem \ref{teorema do M2KGRADINS} were determined in \cite{KA}. As a consequence of Theorem \ref{teorema do M2KGRADINS} and Remark \ref{r13}, there are four gradings we need to consider in this paper. This is done in the remainder of this section.

\subsection{Elementary \texorpdfstring{$\mathbb{Z}_2$}{Z2}-grading on \texorpdfstring{$M_2(\mathbb{F})$}{M2(F)}}

In this section we determine a basis for the graded identities of $M=M_2(\mathbb{F})$ with the following elementary $\mathbb{Z}_2$-grading:

\begin{equation}\label{El2}
			M_0 =
			\begin{pmatrix}
				\mathbb{F} & 0 \\
				0 & \mathbb{F}
			\end{pmatrix},
			\quad
			M_1 =
			\begin{pmatrix}
				0 & \mathbb{F} \\
				\mathbb{F} & 0
			\end{pmatrix}.
\end{equation}

The basis for the graded identities of $M$ that we present next was determined by P. Koshlukov and S. S. Azevedo in \cite{KA} when $\mathrm{char}\,\mathbb{F}\neq 2$. The following notation will be used in this subsection:

\begin{notation}
We denote by $y_i$ and $z_i$ the indeterminates $x_{i,0}$ and $x_{i,1}$  of $X_{G}$, respectively.
\end{notation}

\begin{theorem}
Let $\mathbb{F}$ be a field of characteristic different from $2$ with $q$ elements. The ideal of graded identities of the algebra $M_2(\mathbb{F})$ with the grading (\ref{El2}) is generated by the polynomials
	 $$y_{1}^{q}-y_{1}=0$$
     and 
	 $$(X_{1}-X_{1}^{q})(X_{2}-X_{2}^{q^{2}})(1-[X_{1},X_{2}]^{q-1})=0,$$  where $X_{i}=z_{i}+y_{i}$ for $i=1,2$.
\end{theorem}

Next we determine a basis for the graded identities of $M_2(\mathbb{F})$, with the grading (\ref{El2}), under the assumption that $\mathbb{F}$ has characteristic $2$. For this reason, we adopt the following convention:
\begin{center}
    Until the end of this subsection we assume that $\mathrm{char}\,\mathbb{F}=2$ and the grading group is $\mathbb{Z}_2$.
\end{center}

\begin{lemma}\label{lema da variedade}
	Let $\mathbb{F}$ be a finite field of characteristic $2$ and let $q=|\mathbb{F}|$. Let $C=M_{n}(D)$ be a finite-dimensional graded simple algebra in the variety determined by $y_1^{q}-y_1=0$. Then $C$ is isomorphic to $\mathbb{F}$, $ M_{2}(\mathbb{F})$ with the grading $M$ in (\ref{El2}) or to $ \mathbb{F}[\mathbb{Z}_{2}] $ with the canonical grading. 
\end{lemma}
\begin{proof}
	We prove first that $n\leq 2$. To do that, assume, on the contrary, that $n\geq 3$. If $c\in\{E_{12},E_{23}\}\cap C_0$ then $c^{2}=0$ and $c^{q}=c$, hence $c=0$, which is a contradiction. Then $E_{12},E_{23}\in C_{1}$, hence $E_{13}=E_{12}\cdot E_{23}\in C_{0}$ which again leads to a contradiction.
	Note that $y^{q}-y=0$ is also an identity for $D_{0}$. Together with Wedderburn's Theorem on finite division rings this implies that $D_{0}=\mathbb{F}$. 
Let $T=\mathrm{supp}\, D$. We consider first the case $T=\mathbb{Z}_2$. Let $u'\in D_{1}-\{0\}$. Note that $(u')^{2}\in D_0=\mathbb{F}$. Moreover, since $\mathrm{char}\,\mathbb{F}=2$, there exists $\lambda\in\mathbb{F}^{\times}$ such that $(u')^{2}=\lambda^{2}$. Let $u=\lambda^{-1}u^{'}$. Then $u\in D_{1}$ and $u^{2}=1$. Hence, $D=\mathbb{F}\oplus \mathbb{F}u\simeq \mathbb{F}[\mathbb{Z}_2]$. If $n=2$, then $E_{12}\in C_{1} $ and $uE_{12}\in C_{0}$. Then $0=(uE_{12})^{q}=uE_{12}$, a contradiction. Hence, $n=1$ and $ C\simeq \mathbb{F}[\mathbb{Z}_{2}] $. Now assume that $T=\{0\}$. In this case, $D=\mathbb{F}$. Since $n\leq 2$, we conclude that $C= \mathbb{F}$ or $C= M_{2}(\mathbb{F})$. The grading on $C= M_{2}(\mathbb{F})$ is elementary and the identity $y^{q}-y=0$ implies it is the one in (\ref{El2}).
\end{proof}




\begin{lemma}\label{decomposição COM unidade}
	Let $\mathbb{F}$ be a finite field of characteristic $2$ and let $q=|\mathbb{F}|$. Let $A=A_{0}\oplus A_{1}$ be a unital $\mathbb{Z}_{2}$-graded algebra of finite dimension. Moreover, assume that $y^{q}-y=0$ is a graded identity for $A$. Then $A=B\dotplus J^{gr}(A)$, where $B=B_{1}\oplus\cdots \oplus B_{r}$ and $[B_{i}]\in \{[\mathbb{F}],[\mathbb{F}[\mathbb{Z}_{2}]],[M]\}$ for all $i=1,\dots, r$.
\end{lemma}
\begin{proof}
	Let $J=J^{gr}(A)$. It follows from Theorem \ref{decomposição estrutural de álgebras semi primas graduadas}  and Lemma \ref{lema da variedade} that
	$$A/J= L_{1}\oplus L_{2}\oplus\cdots\oplus L_{r},$$ where $L_i\cong \mathbb{F}$ or $L_i\cong M$, or $ L_i\cong \mathbb{F}[\mathbb{Z}_{2}] $ for all $i=1,\dots, r$.
	 Let $f_{1},\ldots,f_{r}\in Q= A/J$, where $f_{i}$ is the unity of $L_{i}$ for $i=1,\dots, r$. Lemma \ref{lifting} implies that there exist homogeneous orthogonal idempotents $e_{1},\ldots,e_{r}\in A$ such that $\overline{e_i}=f_i$ for all $i=1,\dots, r$. For every $1\leq t\leq r$, we have $f_{t}Qf_{t}=L_{t}$, hence an element of $L_t$ can be lifted to an element of $e_tAe_t$. 
	 
	 
	 If $L_{t}\cong \mathbb{F}$, let $B_t=\mathbb{F}e_{t}\simeq L_t$.	 If $L_{t}\cong M_{2}(\mathbb{F})$, Lemma \ref{lifting} implies that there exist homogeneous matrix units $x^{(t)}_{ij}\in e_tAe_t$, where $1\leq i,j \leq 2$, that lift the matrix units of $L_t$. Then let $B_{t}:=\sum_{i,j}\mathbb{F}x_{ij}^{(t)}\simeq L_{t}$, as graded algebras. Now assume that $L_{t}\cong \mathbb{F}[\mathbb{Z}_{2}]$. Let $ u_{t}\in (L_{t})_{1}$ be an element such that $u_{t}^{2}=f_{t}$. Let $U_{t}$ be a homogeneous element of $e_tAe_t$ such that $\overline{U_t}=u_t$. Note that $e_{t}U_{t}=U_{t}e_{t}=U_{t}$. We claim that $U_{t}^{2}=e_{t}$. Indeed, we have
	 $$\overline{(U_{t}^{2}-e_{t})}=\overline{U_{t}}^{2}-\overline{e_{t}}=u_{t}^{2}-f_{t}=0.$$
	Hence, $U_{t}^{2}-e_{t}\in J$. Note that $U_{t}^{2}-e_{t}$ lies in $J_{0}$ and therefore is nilpotent. Hence, the identity $y^{q}-y=0$ implies that $U_{t}^{2}-e_{t}=0$. Let $B_{t}=\mathbb{F}e_{t}+\mathbb{F}U_{t}$. Note that $B_{t}$ is a subalgebra of $e_tAe_t$ for all $t=1,\dots, r$. Hence $B=B_{1}\oplus\cdots \oplus B_{r}$ is a direct sum, each $B_t$ is an ideal of $B$ and $B$ is a graded subalgebra of $A$.
	 
	From our construction of $B_t$, the restriction of the quotient map $A\rightarrow A/J$ yields an isomorphism from $B_t$ to $L_t$. Hence, $x\mapsto x+J$ is an isomorphism from $B$ to $A/J$. As a consequence $B\cap J=0$. Therefore, $A=B\dotplus J$.
\end{proof}

We are now able to remove the hypothesis of the existence of a unit in \textbf{Lemma \ref{decomposição COM unidade}}.

\begin{theorem}\label{teorema do levantamento}
	Let $\mathbb{F}$ be a finite field of characteristic $2$ with $q$ elements, and let $A=A_{0}\oplus A_{1}$ be a finite-dimensional associative $\mathbb{Z}_{2}$-graded algebra satisfying $y^{q}-y=0$ for every $y\in A_{0}$. Then $A=B\dotplus J^{gr}(A)$, where $B=B_{1}\oplus\cdots \oplus B_{r}$ and the $B_{i}$ are graded simple algebras and ideals of $B$. More specifically, $[B_{i}]\in \{[\mathbb{F}],[\mathbb{F}[\mathbb{Z}_{2}]],[M]\}$.
\end{theorem}
\begin{proof}
	Let $ A^{\sharp}$ be the unitization of $A$. Note that for every $a\in A_0$ and every $\lambda\in \mathbb{F}$ we have \[(\lambda+a)^q=\lambda^q+a^q=\lambda+a,\] hence $y^q-y$ is also a graded identity for $A^{\sharp}$. It follows from Lemma \ref{decomposição COM unidade} that $A^{\sharp}=B'\dotplus J^{gr}(A^{\sharp})$, where $B'$ is a direct sum of graded simple algebras. In particular $J^{gr}(B')=0$. Lemma \ref{jgrunitization} implies that $J^{gr}(A^{\sharp})=J^{gr}(A)$; therefore, $A^{\sharp}=B'\dotplus J^{gr}(A)$.
	
	Consider the map $\varphi:B'\mapsto \mathbb{F}$, which for each $x\in B'$, written as $\lambda1+a$, where $a\in A$, is given by $\varphi(x)=\lambda$. Note that $\varphi$ is an epimorphism of graded algebras. Indeed, it is a homomorphism. If it is not surjective, then it is the zero map. In this case $B'\subseteq A$, hence \[A^{\sharp}=B'\dotplus J^{gr}(A)\subseteq A,\] a contradiction. The kernel of $\varphi$ is $B=B'\cap A$ and we have $B'=\mathbb{F}\dotplus B$. This implies that $B'$ is the unitization of $B=B'\cap A$. Lemma \ref{jgrunitization} implies that $J^{gr}(B)=J^{gr}(B')=0$. Therefore, Theorem \ref{decomposição estrutural de álgebras semi primas graduadas} implies that there exists a decomposition $B=B_{1}\oplus\cdots \oplus B_{r}$ where the $B_{i}$ are graded simple algebras and ideals of $B$. Note that we have a direct sum $B\dotplus J^{gr}(A)$ and this is a subspace of $A$ of dimension equal to $\mathrm{dim}_{\mathbb{F}}\, A$, hence \[A=B\dotplus J^{gr}(A)=B_{1}\oplus\cdots \oplus B_{r}\dotplus J^{gr}(A).\]
    Lemma \ref{lema da variedade} implies that $[B_{i}]\in \{[\mathbb{F}],[\mathbb{F}[\mathbb{Z}_{2}]],[M]\}$.
\end{proof}

\begin{theorem}\label{13}
Let $\mathbb{F}$ be a field of characteristic $2$ with $q$ elements. The ideal of graded identities of the algebra $M_2(\mathbb{F})$ with the grading (\ref{El2}) is generated by the polynomials
\begin{equation}\label{yqy}
	 y_{1}^{q}-y_{1}=0
\end{equation}
and
\begin{equation}\label{ordm2}
(X_{1}-X_{1}^{q})(X_{2}-X_{2}^{q^{2}})(1-[X_{1},X_{2}]^{q-1})=0,  
\end{equation}
 where $X_{i}=z_{i}+y_{i}$ with $||y_{i}||=0$  and  $||z_{i}||=1.$
\end{theorem}

 \begin{proof}
Let $\mathfrak{V}$ denote the variety determined by the identities (\ref{yqy}) and (\ref{ordm2}). Notice that $\mathrm{Var}(M)\subseteq \mathfrak{V}$. The result follows if we prove the reverse inclusion. Identity (\ref{ordm2}) and Corollary \ref{varsdirr} imply that to prove the desired inclusion it is sufficient to prove that every finite graded subdirectly irreducible algebra in $\mathfrak{V}$ lies in $\mathrm{Var}(M)$. We prove, more generally, that each such algebra is isomorphic to a subalgebra of $M$.

Let $A=B_{1}\oplus\cdots\oplus B_{s}\dotplus J^{gr}(A)$ be a finite graded subdirectly irreducible algebra in $ \mathfrak{V}$, where $J^{gr}(A)\neq 0$ and each $B_{i}$ is graded simple; moreover, we assume that $s>0$. 
 Henceforth, denote $N=J^{gr}(A)$. Note that $N$ is a nilpotent algebra in $\mathfrak{V}$. Then $N_0$ is nilpotent and satisfies (\ref{yqy}), hence $N_0=0$. As a consequence $N=N_1$ and $N^2=0$. Hence, $A_{1}=(A_{1}\cap B)\dotplus N$, where $B=B_{1}\oplus\cdots\oplus B_{s}$. We have $A_1N=NA_1=0$; therefore, if $x\in A_1\cap B$ then $\langle x \rangle\subseteq B$, where $\langle x \rangle$ is the ideal generated by $x$. As a consequence $\langle x\rangle \cap N=0 $. Since $A$ is graded subdirectly irreducible and $N\neq 0$, we conclude that $x=0$. Therefore, $A_{1}=N$ and $A_{0}=B$. As a consequence the odd component of $B$ is $0$ and Theorem \ref{teorema do levantamento} implies that $B_i\cong \mathbb{F}$ for all $i=1,\dots, s$.
  
  Note that if $AN=NA=0$ then $B$ is a non-zero graded ideal such that $B\cap N=0$. This is a contradiction since $A$ is subdirectly irreducible. Therefore, $AN\neq 0$ or $NA\neq 0$. Assume, without loss of generality, that $AN\neq0$. Let $e_{i}$ be the unity of $B_{i}$ for $i=1,\dots, s$. Note that $e_{1}N,\ldots,e_{s}N$ are graded ideals; moreover, $e_iN\cap e_jN=0$ whenever $i\neq j$. Therefore only one of these ideals, say $e_{1}N$, is different from zero. Hence, $B_{i}N=0$ for all $i=2, \dots, s$.

  We claim that $N=e_1N$. Indeed, note that $I=\{n-e_{1}n\mid n\in N\}$ is a graded ideal of $A$ such that $I\cap e_1N=0$. Since $A$ is subdirectly irreducible, this implies that $I=0$. We will prove that $A$ is isomorphic to a subalgebra of $M$. We remark that if $NA\neq 0$ then $Ne_j=N$ for some $j$. We may assume in this case that $j\leq 2$. There are three possible cases.

    \textbf{Case 1:} $NA=0$. In this case $s=1$. Indeed, assume that $s\geq 2$. Note that $B_2N=B_2e_1N=0=NB_2$ and $B_iB_2=0=B_2B_i$ for all $i\neq 2$. Hence, $B_{2}$ is a graded ideal such that $B_{2}\cap N=0$, which is a contradiction.
    Hence, we conclude that $s=1$. We claim that $N$ has dimension $1$ over $\mathbb{F}$. Indeed, since $B_{i}=\mathbb{F}e_i$, if $x,y\in N$, then $I=span_{\mathbb{F}}\{x\}$ and $I'=span_{\mathbb{F}}\{y\}$ are graded ideals of $A$. If $x$ and $y$ are linearly independent, we have $I\cap I'=0$, a contradiction. 
    
    Let $n\in N\setminus \{0\}$. Then $\beta=\{e_{1},n\}$ is a basis of $A$ such that $||e_{1}||=0$, $||n||=1$ and
    $$e_{1}n=n,\mbox{ }ne_{1}=0,\mbox{ }e_{1}^{2}=e_{1}\mbox{ and }n^{2}=0.$$
    The map     
    \begin{align*}
    	\varphi: A&\longrightarrow M\\
    	\lambda e_{1}+\mu n&\mapsto\varphi(\lambda e_{1}+\mu n)=\begin{pmatrix}
    		\lambda & \mu\\
    		0 & 0
    	\end{pmatrix}
    \end{align*}
 is a monomorphism of graded algebras.
 
  
 \textbf{Case 2:} $NA\neq 0$ and $Ne_{1}=N$. Since $N=e_{1}N$, we conclude that $N=e_{1}Ne_{1}$. Hence $B_{i}N=NB_{i}=0$ for all $i=2,\dots, s$. Thus, if $s\geq 2$, then $B_{2}$ is a graded ideal of $A$ and $N\cap B_2=0$. This is a contradiction. Since $A$ is graded subdirectly irreducible, we have $s=1$ and $B_1=\mathbb{F} e_1$. We claim that $N$ has dimension $1$ over $\mathbb{F}$. Indeed, assume that $x, y\in N$. Then $I=\mathbb{F}x$ and $I'=\mathbb{F}y$ are graded ideals of $A$. If $x$ and $y$ are linearly independent, then $I\cap I'=0$. Therefore, $N$ has dimension $1$ over $\mathbb{F}$.
 
 Let $n\in N\setminus\{0\}$. Then $A$ has a basis $\beta=\{e_{1},n\}$ such that  $||e_{1}||=\overline{0}$, $||n||=\overline{1}$ and 
$$e_{1}n=n,\mbox{ }ne_{1}=n,\mbox{ }e_{1}^{2}=e_{1}\mbox{ and }n^{2}=0.$$

Therefore, the map
 \begin{align*}
	\varphi: A&\longrightarrow M\\
	\lambda e_{1}+\mu n&\mapsto\varphi(\lambda e_{1}+\mu n)=\begin{pmatrix}
		\lambda & \mu\\
		0 & \lambda
	\end{pmatrix}
\end{align*}
is a monomorphism of graded algebras.

\textbf{Case 3:} $NA\neq 0$ and $Ne_{2}=N$. Since $N=e_{1}N$, we conclude that $N=e_{1}Ne_{2}$. We have $B_{i}N=NB_{i}=0$ for $i\geq 3$. If $s\geq 3$ then $B_3$ is a graded ideal of $A$ such that $N\cap B_3=0$, a contradiction. Hence $s=2$.  Note that $NB_{1}=B_{2}N=0$. We claim that $N$ has dimension $1$ over $\mathbb{F}$. The argument is analogous to the previous cases as the subspace generated by an element of $N$ is a graded ideal of $A$. Let $n$ be a generator of $N$. Then $A$ has a basis $\beta=\{e_{1},e_{2},n\}$ such that $||e_{1}||=||e_{2}||=0$, $||n||=1$, and moreover
$$e_{1}e_{2}=e_{2}e_{1}=ne_{1}=e_{2}n=n^{2}=0, \mbox{ }e_{1}^{2}=e_{1},\mbox{ }e_{2}^{2}=e_{2}\mbox{ and  }e_{1}n=ne_{2}=n.$$

Therefore, the map
\begin{align*}
	\varphi: A&\longrightarrow M\\
	\lambda e_{1}+\alpha e_{2}+\mu n&\mapsto\varphi(\lambda e_{1}+\alpha e_{2}+\mu n)=\begin{pmatrix}
		\lambda & \mu\\
		0 & \alpha
	\end{pmatrix}
\end{align*}
is a monomorphism of graded algebras.

Hence, we conclude that a subdirectly irreducible algebra $A$ in the variety $\mathfrak{V}$ is isomorphic to a subalgebra of $M$ if $N\neq 0$ and $s\geq 1$. It remains to consider the cases $N=A$ and $N=0$.

Assume first that $A=N$. Then $A=N=N_1$ and $A^{2}=0$. In this case $A$ has dimension $1$ over $\mathbb{F}$ because it is graded subdirectly irreducible. Then $A\cong\mathbb{F}E_{12}$ is isomorphic to a subalgebra of $M$.

Now assume that $N=0$. Then $A=B=B_{1}\oplus\cdots\oplus B_{s}$. Since $A$ is graded subdirectly irreducible we have $s=1$. Moreover $A$ is isomorphic as a graded algebra to one of the algebras $\mathbb{F}$, $M$ or $\mathbb{F}[\mathbb{Z}_{2}]$. The first two algebras are isomorphic to subalgebras of $M$. Let $\{1,u\}$ be the canonical basis for $\mathbb{F}[\mathbb{Z}_{2}]$, i.\,e., $1$ is the unity and $u^2=1$. Then the map
\begin{align*}
	\varphi: \mathbb{F}[\mathbb{Z}_{2}]&\longrightarrow M\\
	\lambda\cdot 1+\alpha\cdot u&\mapsto\varphi(\lambda\cdot 1+\alpha\cdot u)=\begin{pmatrix}
		\lambda & \alpha\\
		\alpha & \lambda
	\end{pmatrix}
\end{align*}
is a monomorphism of graded algebras. 
\end{proof}

\subsection{Elementary Fine Grading}

Let $G=\langle g \rangle$ be a cyclic group with $|G|>2$. We do not assume that $G$ has finite order and use the multiplicative notation. In this section $\widetilde{M}$ denotes the algebra $M_2(\mathbb{F})$ with the grading with support $\{e,g, g^{-1}\}$ and

\begin{equation}\label{Elfin}
    \begin{split}
			&\widetilde{M}_e =
			\begin{pmatrix}
				\mathbb{F} & 0 \\
				0 & \mathbb{F}
			\end{pmatrix},
			\quad
			\widetilde{M}_g =
			\begin{pmatrix}
				0 & \mathbb{F} \\
				0 & 0
			\end{pmatrix},
			\quad
			\widetilde{M}_{g^{-1}} =
			\begin{pmatrix}
				0 & 0 \\
				\mathbb{F} & 0
			\end{pmatrix}.			
    \end{split}
\end{equation}



In order to describe the basis for the graded identities of $\tilde{M}$ we establish the following notation:

\begin{notation}
Until the end of this subsection we denote by $y_i$, $z_i$ and $w_i$ the indeterminates $x_{i,e}$, $x_{i,g}$ and $x_{i,g^{-1}}$ of $X_{G}$, respectively.
\end{notation}

We define the following polynomials:
\begin{equation}\label{idelfin}
\begin{split}
f_{1}&=y_{1}^{q}-y_{1},\\
f_{2}&=(X_{1}-X_{1}^{q})(X_{2}-X_{2}^{q^{2}})(1-[X_{1},X_{2}]^{q-1}), \mbox{ where }X_{i}=y_{i}+z_{i}+w_{i},\\
f_{3}&=z_{1}z_{2},\\
f_{4}&=w_{1}w_{2}.
\end{split}
\end{equation}

\begin{lemma}\label{lema de como é o Mn(D) pra levantar e,g,g^(-1)}
	Let $\mathbb{F}$ be a finite field with $q$ elements, and let $G=\langle g \rangle$ be a cyclic group of order greater than $2$. Let $C=M_{n}(D)$ be a finite-dimensional graded simple algebra in the variety determined by the polynomials $f_{1},f_{2},f_{3},f_{4}$ in (\ref{idelfin}) and $x_{1,h}$, for all $h\notin \{e,g,g^{-1}\}$. Then $C$ is isomorphic to $\mathbb{F}$ or to $M_{2}(\mathbb{F})$ with the grading $\widetilde{M}$ in (\ref{Elfin}).
\end{lemma}
\begin{proof}
	Note that $D_e$ satisfies the identity $f_{1}$; therefore $D_{e}=\mathbb{F}$. Moreover, $x_{1,h}$ is an identity for $D$ for all $h\notin \{e,g,g^{-1}\}$; hence $T:=\mathrm{supp}\,D$ is a subgroup of $G$ contained in $\{e,g,g^{-1}\}$. Hence, $T=\{e\}$ or $T=\{e,g,g^{-1}\}$. The latter cannot occur because $D$ satisfies the identities $f_{3}$ and $f_{4}$. Therefore, $T=\{e\}$ and $D=\mathbb{F}$. Hence, $C=M_n(\mathbb{F})$ with an elementary grading. Identity $f_{2}$ implies that $n\leq 2$. If $n=1$ then $C\cong \mathbb{F}$. Now assume that $n=2$. The identities $f_1$ and $x_{1,h}$, $h\in G\setminus \{e,g,g^{-1}\}$ imply that the support of the grading on $C$ is $\{e, g, g^{-1}\}$. Moreover, $f_1$ implies that the grading on $C$ is not the trivial grading. Note that $||E_{12}||=||E_{21}||^{-1}$. Hence, if $||E_{12}||=g$ then $C=\widetilde{M}$. If $||E_{12}||=g^{-1}$ then the map $E_{ij}\mapsto E_{\sigma(i)\sigma(j)}$, where $\sigma$ is the transposition $(12)$ of the permutation group $S_2$, is an isomorphism from $C$ to $\widetilde{M}$.
\end{proof}

\begin{lemma}\label{nome aleata}
	Let $\mathbb{F}$ be a finite field with $q$ elements, and let $G=\langle g \rangle$ be a cyclic group of order greater than $2$. Let $A$ be a finite-dimensional unital $G$-graded algebra that satisfies the identities $f_{1},f_{2},f_{3},f_{4}$ in (\ref{idelfin}) and $x_{1,h}$ for $h\in G \setminus\{e,g, g^{-1}\}$. Then there exists a graded subalgebra $B$ of $A$, such that $B=B_{1}\oplus\cdots\oplus B_{s}$, $s\geq 0$, each $B_{i}$ is isomorphic as a graded algebra to $\mathbb{F}$ or to $\widetilde{M}$, and $A=B\dotplus J^{gr}(A)$.
\end{lemma}
\begin{proof}
	Theorem \ref{decomposição estrutural de álgebras semi primas graduadas} implies that $A/J^{gr}(A)\cong M_{n_{1}}(D_{1})\oplus\cdots\oplus M_{n_{s}}(D_{s})$. Note that each $M_{n_{i}}(D_{i})$ satisfies the hypothesis of Lemma \ref{lema de como é o Mn(D) pra levantar e,g,g^(-1)}; therefore, each $M_{n_{i}}(D_{i})$ is isomorphic to $\mathbb{F}$ or to $\widetilde{M}$. The result now follows from an argument analogous to the one used in the proof of Lemma~\ref{decomposição COM unidade}.
\end{proof}

\begin{theorem}\label{decompelfin}
		Let $\mathbb{F}$ be a finite field with $q$ elements, and let $G=\langle g \rangle$ be a cyclic group of order greater than $2$. Let $A$ be a finite-dimensional $G$-graded algebra that satisfies the identities $f_{1},f_{2},f_{3},f_{4}$ in (\ref{idelfin}) and $x_{1,h}$ for $h\in G \setminus\{e,g, g^{-1}\}$. Then there exists a graded subalgebra $B$ of $A$, such that $B=B_{1}\oplus\cdots\oplus B_{s}$, $s\geq 0$, each $B_{i}$ is isomorphic as a graded algebra to $\mathbb{F}$ or to $\widetilde{M}$, and $A=B\dotplus J^{gr}(A)$.
\end{theorem}
\begin{proof}
	We claim that the unitization $A^{\sharp}$ satisfies the identities $f_{1},f_{2},f_{3},f_{4}$ and $x_{1,h}$ for $h\in G \setminus\{e,g, g^{-1}\}$. The proof that $A^{\sharp}$ satisfies $f_1$ is analogous to the one in the proof of Theorem \ref{teorema do levantamento}. The result is clear for $f_{3},f_{4}$ and $x_{1,h}$. To verify the result for $f_{2}$ it is sufficient to note that for every $a, b\in A^{\sharp}$ and every $\lambda,\mu \in \mathbb{F}$ we have \[(\lambda+a)-(\lambda+a)^{q^n}=a-a^{q^n}\] and \[[\lambda+a, \mu +b]=[a,b].\]
    The rest of the proof is analogous to that of Theorem \ref{teorema do levantamento}.
\end{proof}

\begin{theorem}
Let $\mathbb{F}$ be a field with $q$ elements. The ideal of graded identities of the algebra $M_2(\mathbb{F})$ with the grading (\ref{Elfin}) is generated by the polynomials:
\begin{equation}\label{varefin}
\begin{split}
&f_{1}=y_{1}^{q}-y_{1},\\
&f_{2}=(X_{1}-X_{1}^{q})(X_{2}-X_{2}^{q^{2}})(1-[X_{1},X_{2}]^{q-1}), \mbox{ where } X_{i}=y_{i}+z_{i}+w_{i},\\
&f_{3}=z_{1}z_{2},\\
&f_{4}=w_{1}w_{2},\\
&x_{1,h},\,\mbox{ where } h\in G\setminus \{e,g,g^{-1}\}.
\end{split}
\end{equation}
\end{theorem}
\begin{proof}

Let $\mathfrak{B}$ be the variety determined by the polynomials in (\ref{varefin}). We have $\mathrm{Var}(\widetilde{M})\subseteq \mathfrak{B}$. Every algebra in $\mathfrak{B}$ satisfies $f_2$ and $x_{1,h}$ for all $h\in G\setminus \{e,g,g^{-1}\}$. By Corollary \ref{varsdirr}, it suffices to consider finite algebras in $\mathfrak{B}$ that are graded and subdirectly irreducible. It remains to verify that each such algebra lies in $\mathrm{Var}(\widetilde{M})$.



Let $A$ be a finite graded subdirectly irreducible algebra in $\mathfrak{B}$. Theorem \ref{decompelfin} implies that $A=B_{1}\oplus\cdots\oplus B_{r}\dotplus N$, where $B_i$ is isomorphic to $\mathbb{F}$ or to $\widetilde{M}$ for all $i=1,\dots, r$ and $N=J^{gr}(A)$.

We claim that $N^2=0$. Indeed, $N=N_{e}\oplus N_{g}\oplus N_{g^{-1}}$. Identity $f_1$ implies that $N_{e}=0$. Hence, $N=N_{g}\oplus N_{g^{-1}} $ and $N_{g}N_{g^{-1}},N_{g^{-1}}N_{g}\subseteq N_{e}=0$. Thus, identities $f_3$ and $f_4$ imply that $N^{2}=0$ and the claim is proved. As a consequence, if $r=0$, i.\,e., $A=N$ is nilpotent then $A$ has dimension $1$ over $\mathbb{F}$ and it is isomorphic to a subalgebra of $\widetilde{M}$.

Now assume that $N=0$. Then $r=1$ and $A$ is isomorphic to $\mathbb{F}$ or $\widetilde{M}$, hence it is isomorphic to a graded subalgebra of $\widetilde{M}$. 
 
 Finally, assume that $r\geq 1$ and that $N\neq 0$. We know that $N_{e}=0$. Therefore, $N\subseteq S:= A_{g}\oplus A_{g^{-1}}$.  


We claim that $S\cap B=0$, where $B=B_{1}\oplus\cdots\oplus B_{r}$. Let $x$ be a homogeneous element in $S\cap B$ and let $u\in N$. Then $xu=0=ux$. Hence $I\subseteq B$, where $I$ is the ideal generated by $x$. Therefore $I\cap N=0$. As a consequence $x=0$. Note that 
 \[S\cap(B\dotplus N)=(S\cap B)\oplus N.\] Since $A=B\dotplus N$ and $S\cap B=0$, we conclude that $N= S$. Since $J(A_e)=0$, due to identity $f_1$, we conclude that $A_{e}=B$. Moreover, the fact that $f_1$ is an identity for $A_{e}$ and Lemma \ref{nome aleata} imply that $B_{i}\simeq \mathbb{F}$, for all $i=1,\ldots,r$.

If $AN=NA=0$ then $B$ is a non-zero graded ideal such that $B\cap N=0$. This is a contradiction, since $A$ is subdirectly irreducible. Therefore, $AN\neq 0$ or $NA\neq0$. We may assume without loss of generality that $AN\neq 0$. Let $e_{i}$ be the unity of $B_{i}$ for $i=1,\dots, r$.

Note that the $e_{i}N$, $i=1,\dots, r$ are graded ideals of $A$. Since $A$ is subdirectly irreducible, only one of them is non-zero. Hence we may assume that $e_{1}N\neq 0$ and $e_iN=0$ for $i>1$. As in the proof of Theorem \ref{13} we conclude that $N=e_{1}N$. Moreover, only one of the ideals $Ne_{j}$ is different from zero. We may assume that $j\leq 2$. There are three cases to consider.

\textbf{Case 1:} $NA=0$. We conclude, as in Case 1 of the proof of Theorem \ref{13} that $r=1$ and $\mathrm{dim}_{\mathbb{F}}\,N=1$. Let $n$ be a non-zero element of $N$. Then $\beta=\{e_{1},n\}$ is a basis of $A$ such that $e_{1}n=n$, $ne_{1}=n^{2}=0$ and $e_{1}^{2}=e_{1}$. Hence, if $N=N_{g}$, the map
\begin{align*}
	\varphi: A&\longrightarrow \widetilde{M}\\
	\lambda e_{1}+\mu n&\mapsto \varphi(\lambda e_{1}+\mu n)=\begin{pmatrix}
		\lambda & \mu\\
		0 & 0
	\end{pmatrix}
\end{align*} 
is a monomorphism of graded algebras.
 If $N=N_{g^{-1}}$ then 
\begin{align*}
	\varphi: A&\longrightarrow \widetilde{M}\\
	\lambda e_{1}+\mu n&\mapsto \varphi(\lambda e_{1}+\mu n)=\begin{pmatrix}
		0 & 0\\
		\mu & \lambda
	\end{pmatrix}
\end{align*} 
is a monomorphism of graded algebras.

\textbf{Case 2:} $Ne_{1}\neq 0$. Then as in Case 2 of Theorem \ref{13} we have $N=e_{1}Ne_{1}$,  $r=1$ and $\mathrm{dim}_{\mathbb{F}}\, N=1 $. Let $n$ be a non-zero element of $N$. Then $\beta=\{e_{1},n\}$ is a basis of $A$ such that $e_{1}^{2}=e_{1}$, $e_{1}n=ne_{1}=n$ and $n^{2}=0$. Hence, if $N=N_g$ the map
\begin{align*}
	\varphi: A&\longrightarrow \widetilde{M}\\
	\lambda e_{1}+\mu n&\mapsto \varphi(\lambda e_{1}+\mu n)=\begin{pmatrix}
		\lambda & \mu\\
		0 & \lambda
	\end{pmatrix}
\end{align*}
is a monomorphism of graded algebras.
  
If $N=N_{g^{-1}}$ then the map
\begin{align*}
	\varphi: A&\longrightarrow \widetilde{M}\\
	\lambda e_{1}+\mu n&\mapsto \varphi(\lambda e_{1}+\mu n)=\begin{pmatrix}
		\lambda & 0\\
		\mu & \lambda
	\end{pmatrix}
\end{align*} 
is a monomorphism of graded algebras.

\textbf{Case 3:} $Ne_{2}\neq 0$. In this case $N=e_{1}Ne_{2}$. As in Case 3 of Theorem \ref{13} we conclude that $r=2$ and $\mathrm{dim}_{\mathbb{F}}\, N=1$. Let $n$ be a non-zero element of $N$. Then $\beta=\{e_{1},e_{2},n\}$ is a basis of $A$ such that
  $$e_{1}^{2}=e_{1},\ 
  e_{2}^{2}=e_{2}, \
 e_{1}e_{2}=e_{2}e_{1}=n^{2}=e_{2}n=ne_{1}=0,\ e_{1}n=ne_{2}=n.$$

If $N=N_{g}$ then the map
\begin{align*}
	\varphi: A&\longrightarrow \widetilde{M}\\
	\lambda_{1} e_{1}+\lambda_{2}e_{2}+\mu n&\mapsto \varphi(\lambda_{1} e_{1}+\lambda_{2}e_{2}+\mu n)=\begin{pmatrix}
		\lambda_{1} & \mu\\
		0 & \lambda_{2}
	\end{pmatrix}
\end{align*} 
is a monomorphism of graded algebras.

If $N=N_{g^{-1}}$ then the map
\begin{align*}
	\varphi: A&\longrightarrow \widetilde{M}\\
	\lambda_{1} e_{1}+\lambda_{2}e_{2}+\mu n&\mapsto \varphi(\lambda_{1} e_{1}+\lambda_{2}e_{2}+\mu n)=\begin{pmatrix}
		\lambda_{2} & 0\\
		\mu & \lambda_{1}
	\end{pmatrix}
\end{align*}
is a monomorphism of graded algebras.
\end{proof}




\subsection{Division \texorpdfstring{$\mathbb{Z}_2$}{Z2}-gradings}
In this section we determine a basis for the graded identities of $M_2(\mathbb{F})$ with the two possible division gradings by the group $\mathbb{Z}_2$.

First, we assume that $\mathrm{char}\,\mathbb{F}\neq 2$. The grading for $\widehat{D}=M_2(\mathbb{F})$ that we need to consider is the following:

\begin{equation}\label{D1}
		      \widehat{D}_0 =
			\left\{
			\begin{pmatrix}
				u & v \\
				bv & u
			\end{pmatrix}
			\;\middle|\; u,v \in \mathbb{F}
			\right\},
			\quad
			\widehat{D}_1 =
			\left\{
			\begin{pmatrix}
				u & v \\
				-bv & -u
			\end{pmatrix}
			\;\middle|\; u,v \in \mathbb{F}
			\right\},
\end{equation}
where $b \in \mathbb{F} \setminus \mathbb{F}^2$.	

\begin{notation}
In this subsection we denote by $y_i$ and $z_i$ the indeterminates $x_{i,0}$ and $x_{i,1}$  of $X_{G}$, respectively.
\end{notation}

A basis for the graded identities of $\widehat{D}$ was determined in \cite[Theorem 15]{KA}, which we present next.

\begin{theorem}
Let $\mathbb{F}$ be a field of characteristic different from $2$ with $q$ elements. The ideal of graded identities of the algebra $M_2(\mathbb{F})$ with the grading (\ref{D1}) is generated by the polynomials
\begin{equation*}
\begin{split}
	&y_{1}^{q^{2}}-y_{1}, \mbox{ with } ||y_{1}||=0,\\
	 &z_{1}^{2q-1}-z_{1}, \mbox{ with } ||z_{1}||=1,\\
	 (&X_{1}-X_{1}^{q})(X_{2}-X_{2}^{q^{2}})(1-[X_{1},X_{2}]^{q-1}).
\end{split}
\end{equation*}
     Here, $X_{i}=z_{i}+y_{i}$ with $||y_{i}||=0$  and  $||z_{i}||=1.$
\end{theorem}

Until the end of this subsection we assume that $\mathrm{char}\,\mathbb{F}=2$.

When $\mathrm{char}\,\mathbb{F}=2$, the grading for $\widetilde{D}=M_2(\mathbb{F})$ that we need to consider is the following:

\begin{equation}\label{D2}
    \begin{split}
		     \widetilde{D}_0 &=
		\left\{
		\begin{pmatrix}
			u & u+v \\
			b(u+v) & v
		\end{pmatrix}
		\;\middle|\; u,v \in \mathbb{F}
		\right\},\\        
		\widetilde{D}_1 &=
		\left\{
		\begin{pmatrix}
			bu+v & u \\
			v & bu+v
		\end{pmatrix}
		\;\middle|\; u,v \in \mathbb{F}
		\right\},
    \end{split}
\end{equation}
where $b \in \mathbb{F} \setminus \{ x^2 + x \mid x \in \mathbb{F} \}$.

Note that $\mathbb{L}:=\widetilde{D}_{0}$ is a field with $q^{2}$ elements. Let $P=\begin{pmatrix}
	1& 0\\
	1& 1
	\end{pmatrix}\in \widetilde{D}_{1}$. Note that $P^{2}=1\in\mathbb{F}$. The map $\ell\mapsto P\ell P$ is a non-trivial automorphism of $\mathbb{L}$. Therefore, 
    \begin{equation}\label{Pell}
    P\ell =\ell^{q}P, 
    \end{equation}
    for all $\ell \in\mathbb{L}$.   
    
\begin{theorem}
Let $\mathbb{F}$ be a field of characteristic $2$ with $q$ elements. The ideal of graded identities of the algebra $M_2(\mathbb{F})$ with the grading (\ref{D2}) is generated by the polynomials:
\begin{equation}\label{basisd2}
    \begin{split}
	f_{1}&=y_1^{q^{2}}-y_1,\\
	f_{2}&=z_1^{2q-1}-z_1,\\
	f_{3}&=z_1y_1-y_1^{q}z_1,\\
	f_4&=(X_{1}-X_{1}^{q})(X_{2}-X_{2}^{q^{2}})(1-[X_{1},X_{2}]^{q-1}), 
    \end{split}     
    \end{equation}
    where $X_{i}=z_{i}+y_{i}$ with $||y_{i}||=0$  and $||z_{i}||=1.$
\end{theorem}
\begin{proof}
    We have $\widetilde{D}_{0}=\mathbb{L}$ and $\widetilde{D}_{1}=\mathbb{L}P$. It follows from the fact that $\mathbb{L}$ is a field with $q^2$ elements and from (\ref{Pell}) that the polynomials $f_1$, $f_2$ and $f_3$ in (\ref{basisd2}) are graded identities for $\widetilde{D}=M_2(\mathbb{F})$, with the grading (\ref{D2}). Note that since the underlying algebra is $M_2(\mathbb{F})$ the polynomial $f_4$ is an identity for $\widetilde{D}$. Therefore, $\mathrm{Var}(\widetilde{D})\subseteq \mathfrak{B}$, where $\mathfrak{B}$ is the variety determined by the polynomials in (\ref{basisd2}). 
    The algebras in $\mathfrak{B}$ satisfy $f_4$. Therefore, Corollary~\ref{varsdirr} implies that to prove the reverse inclusion it is sufficient to prove that every finite graded subdirectly irreducible algebra in $\mathfrak{B}$ also lies in $\mathrm{Var}(\widetilde{D})$.
	Let $A$ be a finite graded subdirectly irreducible algebra in $\mathfrak{B}$. Since $J^{gr}(A)$ is a nilpotent algebra in $\mathfrak{B}$, we conclude that $J^{gr}(A)=0$. Hence, Theorem \ref{decomposição estrutural de álgebras semi primas graduadas} together with the fact that $A$ is subdirectly irreducible implies that $A$ is a graded simple algebra.
	Moreover, identities $f_{1}$ and $f_{2}$ imply that $A$ is a graded division algebra. Let $A=D$. Identity $f_1$ implies that $D_0$ is a field extension of $\mathbb{F}$ of degree at most $2$. If $\mathrm{supp}\, D=\{0\}$ then $D$ is a field and $D=\mathbb{F}$ or $D\cong \mathbb{L}$. In either case $D$ is isomorphic to a subalgebra of $\widetilde{D}$ and therefore lies in $\mathrm{Var}(\widetilde{D})$. Now assume that $\mathrm{supp}\, D=\mathbb{Z}_2$. Let $T\in D_{1}\setminus \{0\}$. Identity $f_{3}=z_1y_1-y_1^{q}z_1$ implies that
	\begin{align*}
		T^{3}=T(T^{2})=(T^{2})^{q}T=T^{2q+1}\Longrightarrow T^{3}T^{-1}=T^{2q+1}T^{-1}\Longrightarrow T^{2}=(T^{2})^{q}.
	\end{align*}
	Hence, $T^{2}\in\mathbb{F}$. Since $\mathbb{F}$ is a finite field of characteristic $2$, there exists $\lambda\in\mathbb{F}$ such that $T^{2}=\lambda^{2}$. Let $U=\frac{1}{\lambda}T$. Then $U^2=1$ and $D=D_{0}\oplus D_{0}U$. Note that $Ux=x^qU$ for all $x \in D_0$ because $f_3$ holds. The map $\nu +\mu U\mapsto \psi(\nu)+\psi(\mu)P$, where $\psi: D_{0}\rightarrow \mathbb{L}$ is any embedding, is a monomorphism from $D$ to $\widetilde{D}$. Therefore, $D\in \mathrm{Var}(\widetilde{D})$.
\end{proof}

\subsection{Division \texorpdfstring{$\mathbb{Z}_2\times\mathbb{Z}_2$}{Z2 x Z2}-grading}
Let $\mathbb{F}$ be a field with $q$ elements and $\mathrm{char}\, \mathbb{F}\neq 2$. We determine a basis for the algebra $E=M_2(\mathbb{F})$ with the following grading:
\begin{equation}\label{E}			
			E_{(0,0)} = \mathbb{F} I_2, \quad E_{(1,0)} = \mathbb{F}X, \quad E_{(0,1)} = \mathbb{F}Y, \quad E_{(1,1)} = \mathbb{F}XY,			
\end{equation}
where $X,Y$ are invertible matrices such that			
\[X^2, Y^2 \in \mathbb{F} I_2
\quad \text{and} \quad
XY = -YX.
\]

\begin{notation}
Henceforth, we will denote $G=\mathbb{Z}_2\times \mathbb{Z}_2$ with the multiplicative notation, $a=(1,0)$ and $b=(0,1)$. Moreover, let $\beta:G\times G\rightarrow \mathbb{F}^{\times}$ be the alternating bicharacter such that $\beta(a,b)=-1$.
\end{notation}

\begin{lemma}\label{20}
Let $g\in G$. There exists $s_g\in \{1,-1\}$ such that 
\begin{equation}\label{id1}
x_{1,g}^q-s_gx_{1,g}
\end{equation}
is a graded polynomial identity for $E$.
\end{lemma}
\begin{proof}
Let $P\in E_g\setminus \{0\}$ and let $\lambda$ be an element of $\mathbb{F}$. Note that $\omega_P=P^2\in \mathbb{F}$. We have
$$(\lambda P)^{q}=\lambda^qP^{q-1}P=\lambda(P^{2})^{\frac{q-1}{2}}P=\omega_{P}^{\frac{q-1}{2}}(\lambda P).$$
Let $s_g=\omega_{P}^{\frac{q-1}{2}}$. Note that $\omega_P\neq 0$; therefore, \[s_g^2=\omega_{P}^{q-1}=1.\] As a consequence $s_g\in \{1,-1\}$ and $x_{1,g}^q-s_gx_{1,g}$ is a graded polynomial identity for $E$.
\end{proof} 

\begin{remark}\label{21}
Given $g,h\in G$, it is clear that 
\begin{equation}\label{id2}
f(x_{1,g},x_{2,h})=x_{1,g}x_{2,h}-\beta(g,h) x_{2,h}x_{1,g}
\end{equation} is a graded polynomial identity for $E$.
\end{remark}
 
 \begin{notation}
 Let $n$ be a positive integer and write $n=a_{0}+a_{1}q+\cdots+a_{m}q^{m} $, where $0\leq a_{i}<q$. We define $n_{1}=a_{0}+a_{1}+\cdots+a_{m}$. If $0\leq n_{1}<q$, we set $l(n)=n_{1}$; otherwise, we compute $(n_{1})_{1}$ from $n_{1}$ and repeat the process until the result $l(n)$ satisfies $0\leq l(n) \leq q-1$.
 \end{notation}

 \begin{notation}
 We denote by $I$ the $T_G$-ideal generated by the polynomials in (\ref{id1}) and (\ref{id2}). Moreover, $R=\mathbb{F}\langle X_G\rangle/I$ denotes the relatively free algebra. We denote by $x_{i,g}$ the images of the indeterminates in $R$.
 \end{notation}
 
 \begin{remark}\label{multihom}
 	The identities in (\ref{id1}) imply that $x_{i,g}^n\equiv_I\lambda_{g,n} x_{i,g}^{l(n)}$, for some $\lambda_{g, n}\in \{1, -1\}$. Therefore, if $p$ is a multihomogeneous polynomial in the indeterminates $x_{1, g_1},\ldots,x_{n, g_n}$ of multidegree $(t_{1},\ldots,t_{n})$, then the identities (\ref{id1}) and (\ref{id2}) for $R$ imply that $p\equiv_{I} \lambda x_{1, g_1}^{l(t_{1})}\cdots x_{n, g_n}^{l(t_{n})}$, for some $\lambda\in\mathbb{F}$.
 \end{remark}
 
 

\begin{theorem}
	Let $\mathbb{F}$ be a field of characteristic different from $2$ with $q$ elements. The ideal of graded identities of the algebra $E=M_2(\mathbb{F})$ with the grading (\ref{E}) is generated by the polynomials in (\ref{id1}) and (\ref{id2}). Moreover, if $\leq$ is a total order on $X_G$, then a basis, as a vector space, for the relatively free algebra $\mathbb{F}\langle X_G\rangle/Id_G(E)$ consists of the images of the monomials in the set $$\mathcal{M}=\{M=x_{i_1,g_{i_1}}\cdots x_{i_m,g_{i_m}}\},$$ where $m\geq 1$, $x_{i_1,g_{i_1}} \leq \cdots \leq x_{i_m,g_{i_m}}$, $0\leq \mathrm{deg}_{x_{i,g}}\, M<q$ for all $i\geq 1$ and all $g\in G$.
\end{theorem}
\begin{proof}
	Lemma \ref{20} and Remark \ref{21} imply that $I\subseteq Id_G(E)$. Next we prove the reverse inclusion. Let $p\in Id_G(E)$. Remark \ref{multihom} implies that there exist $M_1,\dots, M_n\in \mathcal{M}$ and $\lambda_1,\dots, \lambda_n\in \mathbb{F}$ such that 
    \begin{equation}\label{p}
    p\equiv_I \sum_i\lambda_iM_i.
    \end{equation}
    Since $p\in Id_G(E)$ and $I\subseteq Id_G(E)$, we conclude that $\sum_i\lambda_iM_i$ is a graded identity for $E$. Then \cite[Proposition 1.2.8]{DF} implies that $\lambda_iM_i\in Id_G(E)$ for all $i=1,\dots, n$. Since $E$ has a division grading, this implies that $\lambda_i=0$ for all $i=1,\dots, n$. Then (\ref{p}) implies that $p\in I$. Therefore, we conclude that $Id_G(E)\subseteq I$. The last assertions follow directly from the equality $Id_G(E)= I$ and the preceding arguments.
\end{proof}

\end{document}